\documentclass{amsart}
\usepackage{graphicx}
\usepackage{amsfonts}
\newtheorem{tm}{Theorem}

\newtheorem{defi}{Definition}

\newtheorem{rem}{Remark}
\newtheorem{rems}{Remarks}
\newtheorem{lm}{Lemma}

\newtheorem{prop}{Proposition}
\newtheorem{nota}{Notation}
\newtheorem{quest}{Question}

\title{Where not to find the spectrum of the partial theta function}
\author{Yousra Gati}
\address{Universit\'e de Carthage, EPT-LIM, Tunisie}
\email{yousra.gati@gmail.com}
\author{Vladimir Petrov Kostov} 
\address{Universit\'e C\^ote d'Azur, 
LJAD, Nice, France}  
\email{vladimir.kostov@unice.fr} 
\date{}
\begin{document} 

\begin{abstract}
  The spectrum of
  Ramanujan's partial theta function
  $\theta (q,x):=\sum _{j=0}^{\infty}q^{j(j+1)/2}x^j$, $q\in \mathbb{D}_1$
  (the unit disk centered at the origin),
  $x\in \mathbb{C}$, is the set of values of the parameter $q$ for which
  $\theta (q,.)$ has a multiple zero.
  We show that there is no spectral value in the set
  $\mathbb{S}\cup \mathbb{D}_{c_0}$, $c_0=0.2078750206\ldots$,
  where $\mathbb{S}$ is the sector
  $\{ 0<|z|<0.6,{\rm arg}(z)\in [\pi /4 ,7\pi /4 ]\}$. There is a single
  spectral value in the set $\mathbb{S}\cup \mathbb{D}_{0.31}$ which equals
  $0.309249\ldots$. For  $q\in \mathbb{S}\cup \mathbb{D}_{c_0}$, the moduli of the zeros of $\theta$ are separated by the negative half-integer powers of $|q|$.\\

{\bf Key words:} partial theta function, Jacobi theta function,
Jacobi triple product\\

{\bf AMS classification:} 26A06
\end{abstract}

\maketitle

\section{Introduction}

We consider S.~Ramanujan's {\em partial theta function}
$\sum_{j=0}^{\infty}q^{j(j+1)/2}x^j$ (see \cite{AnBe} and \cite{Wa}) regarding
$x\in \mathbb{C}$ as a variable and $q\in \mathbb{D}_1$ as a parameter;
$\mathbb{D}_r\subset \mathbb{C}$ stands for the open disk centered
at the origin and of radius $r$. The function's name is chosen because of its
resemblance with the {\em Jacobi theta function}
$\Theta (q,x)=\sum_{j=-\infty}^{\infty}q^{j^2}x^j$, as
$\theta (q^2,x/q)=\sum_{j=0}^{\infty}q^{j^2}x^j$; ``partial'' refers to the
summation being performed only for non-negative indices. We remind that

$$\Theta (q,x^2)=\prod_{m=1}^{\infty}(1-q^{2m})(1+x^2q^{2m-1})(1+x^{-2}q^{2m-1})~~~\,
{\rm (Jacobi~triple~product)}$$
which implies the equalities

\begin{equation}\label{equTheta}
  \Theta^*(q,x):=\Theta (\sqrt{q},\sqrt{q}x)=\sum_{j=-\infty}^{\infty}q^{j(j+1)/2}x^j
=\prod_{m=1}^{\infty}(1-q^m)(1+xq^m)(1+q^{m-1}/x)~.\end{equation}
Thus one can write

\begin{equation}\label{equThetaG}
  \theta (q,x)=\Theta^*(q,x)-G(q,x)~,~~~\,
  G(q,x):=\sum_{j=-\infty}^{-1}q^{j(j+1)/2}x^j=\sum_{j=1}^{\infty}q^{j(j-1)/2}x^{-j}~.
\end{equation}
The partial theta function is of interest in many domains: Ramanujan type
$q$-series (\cite{Wa}), the theory of (mock) modular forms (\cite{BrFoRh}),
statistical physics and combinatorics (\cite{So}), asymptotic analysis
(\cite{BeKi}), asymptotics
and modularity of partial and false theta functions and their
interaction with representation theory and conformal field theory
(\cite{BFM}, \cite{CMW}), relationship between
Appell-Lerch sums and mock theta functions (\cite{EM}),
Artin-Tits monoids (\cite{FGM}), 
quantum many-body systems (\cite{WPG}). Andrews-Warnaar identities for
the partial theta function can be found in \cite{WM}, \cite{Wei}
and~\cite{Sun}. The explicit combinatorial interpretation of the
coefficients of the leading root of $\theta$ as a series in~$q$
is given in~\cite{Pr}.  Pad\'e
approximants of $\theta$ are studied in~\cite{LuSa}.

For each $q$ fixed, the function $\theta$ is an entire function of order $0$, see \cite[formula~(5)]{KoDBAN1}. It has countably-many zeros (denoted by $\xi_j$) which are distinct for $|q|<c_0$, see Notation~\ref{notac0} and part (3) of Remarks~\ref{remstmmain}.

To continue the list of applications of $\theta$ we define its {\em spectrum} as the set
of values of the parameter $q$ for which $\theta (q,.)$ has a multiple zero; its elements are called {\em spectral numbers} or {\em spectral values}.
Recently, the role played by $\theta$ in the theory of
section-hyperbolic polynomials (i.~e. real univariate polynomials with all
roots real negative and with all their truncations having only real
negative roots) was revealed in \cite{Ost}, \cite{KaLoVi} and \cite{KoSh},
based on the classical works of Hardy, Petrovitch and Hutchinson
(see~\cite{Ha}, \cite{Pe} and~\cite{Hu}). The relationship with the
spectrum of $\theta$ was mentioned in~\cite{KoSh}. This stimulated the
study of the analytic properties of $\theta$, see
\cite{KoFAA}--\cite{KoArxiv}.

\begin{rem}\label{remsp}
  {\rm The positive spectral numbers}
    $$0<\tilde{q}_1=0.309249\ldots <
    \tilde{q}_2=0.516959\ldots <\tilde{q}_3=0.630628\ldots <\cdots$$
    {\rm form an increasing sequence tending to
    $1^-$, see the list of the 6-digit truncations of the first $25$ of them
    in~\cite{KoSh}. The spectral number $\tilde{q}_1$ is the closest to $0$.
    The negative spectral numbers
    $\cdots <\bar{q}_2<\bar{q}_1=-0.72713332\ldots <0$
    form a decreasing sequence tending to $-1^+$,
    see~\cite[Table~1]{KoPRSE2}. The
    existence of the complex spectral numbers}

  $$v_{\pm}:=0.4353184958\ldots \pm i\cdot 0.1230440086\ldots$$
  {\rm is justified in~\cite[Proposition~8]{KoAM}.}
  \end{rem}

\begin{nota}\label{notac0}
  {\rm For $q\in \mathbb{D}_1$, we denote by
    $\mathbb{A}_{a,b}(|q|)\subset \mathbb{C}$
    the open annulus $\{ |q|^{-a}<|x|<|q|^{-b}\}$, $0<a<b$, and by
    $\mathbb{S}_{r;\alpha ,\beta}\subset \mathbb{C}$ the sector
    $\{ 0<|x|<r,{\rm arg}(x)\in [\alpha ,\beta ]\}$,
    $0\leq \alpha <\beta \leq 2\pi$. We denote by $c_0:=0.2078750206\ldots$
    the solution to the equation $2\sum_{\nu =1}^{\infty}|q|^{\nu^2/2}=1$. We set
    $\theta_k:=\sum_{j=0}^kq^{j(j+1)/2}x^j$ (the $k$th {\em truncation} of $\theta$)
    and $\theta_k^{\bullet}:=
    \sum_{j=k+1}^{\infty}q^{j(j+1)/2}x^j$.}
\end{nota}

\begin{defi}
  {\rm We say that the zeros $\xi_j$ of the partial theta function  are {\em strongly
    separated
    in modulus} if}
  \begin{equation}\label{equsepar}
    \xi_1\in (\mathbb{D}_{|q|^{-3/2}}\setminus \{ 0\} )~~~\, {\rm and}~~~\,
    \xi_k\in \mathbb{A}_{k-1/2,k+1/2}(|q|)~,~~~\, k\geq 2~.
    \end{equation}
\end{defi}

It is clear that if for $q=q_*$, the zeros of $\theta$ are strongly separated in
modulus, then $q_*$ is not a spectral number. We prove
in Section~\ref{secprtmmain} the following theorem:

\begin{tm}\label{tmmain}
  (1) For $q\in \mathbb{S}_{0.6,\pi /4,7\pi /4}\cup \mathbb{D}_{c_0}$, the zeros of
  $\theta$ are strongly separated in modulus.

  (2) The only spectral number of
  $\theta$ in the set $\mathbb{S}_{0.6,\pi /4,7\pi /4}\cup \mathbb{D}_{0.31}$
  is $\tilde{q}_1$.
  \end{tm}

\begin{rems}\label{remstmmain}
  {\rm (1) Theorem~\ref{tmmain} improves the main result in \cite{KoVJM}
    which says that for $q\in \mathbb{S}_{0.55,\pi /2,3\pi /2}$, the zeros of
    $\theta$ are strongly separated in modulus and that for
    $q\in \mathbb{S}_{0.6,\pi /2,3\pi /2}$ and $k\neq 2$, $3$, conditions
    (\ref{equsepar}) hold true. 
  
  (2) The existence of the following spectral numbers of modulus $<0.6$ explains why
    the sector $\mathbb{S}_{0.6,\pi /4,7\pi /4}$ cannot be replaced by the disk
    $\mathbb{D}_{0.6}$: $\tilde{q}_1$, $\tilde{q}_2$ and
    $v_{\pm}$
    (see Remark~\ref{remsp}). The radius $0.6$ cannot be extended
    much due to the presence of the spectral number $\bar{q}_1$, see
    Remark~\ref{remsp}.
    
    (3) For $q\in \overline{\mathbb{D}_{c_0}}$, the zeros of $\theta$ are strongly separated in modulus, see \cite[Lemma~1]{KoAM}.}
  \end{rems}

{\bf Acknowledgement.} The authors thank Boris Shapiro from the University
of Stockholm for the figure included in this paper. 

\section{Comments}

\subsection{The spectral values of $\theta$ in $\mathbb{D}_{0.6}$}

A natural question to ask is which spectral values of $\theta$
except $\tilde{q}_1$, 
$\tilde{q}_2$ and $v_{\pm}$ belong to the disk
$\mathbb{D}_{0.6}$, because $0.6$ is the radius of the sector
$\mathbb{S}_{0.6,\pi /2,3\pi /2}$, see Theorem~\ref{tmmain}.
To this end one can consider the truncations
$\theta_k$ of $\theta$ and the {\em spectral
values (numbers)} of $\theta_k$, i.~e. the values of $q$ for which $\theta_k$ has a multiple zero. 
The truncations
$\theta_8$, $\theta_9$ and $\theta_{10}$ have multiple zeros for
(but not only for) the following values of~$q$:

$$\begin{array}{cc}0.5374009225\ldots \pm i\cdot 0.1800191987\ldots ~,&
  0.5373319887\ldots \pm i\cdot 0.1803221641\ldots \\ \\ 
  {\rm and}&0.5373387872\ldots \pm i\cdot 0.1803273624\ldots ~.\end{array}$$
We list here the truncations up to $4$ digits of their $5$
respective zeros of least modulus:

$$\begin{array}{lll}-2.0471-i\cdot 1.4799~,&-4.7152+i\cdot 1.7057~,&
  -2.4631+i\cdot 7.6623~,\\ \\
  -2.4639+i\cdot 7.6621~,&1.6159+i\cdot 20.5308~,&\\ \\
  -2.0467-i\cdot 1.4773~,&-4.7087+i\cdot 1.7093~,&-2.4678+i\cdot 7.6608~,\\ \\
  -2.4684+i\cdot 7.6612~,&0.8219+i\cdot 17.0772~,&\\ \\
  -2.0467-i\cdot 1.4772~,&-4.7085+i\cdot 1.7091~,&-2.4680+i\cdot 7.6613~,\\ \\
  -2.4685+i\cdot 7.6611~,&0.7472+i\cdot 17.0808~.\end{array}$$

\begin{defi}\label{defiinterms}
  {\rm We say ``the $j$th and $(j+1)$st zero of $\theta$ (or its truncation
      $\theta_r$) in terms of
      the modulus'' in the sense that there are exactly $j-1$ zeros counted with multiplicity  
    whose moduli are smaller than the moduli of these two, and the moduli
    of all other zeros are larger than their moduli.}
  \end{defi}
The third and fourth (in terms of the modulus) zeros of $\theta_8$, $\theta_9$
and $\theta_{10}$ are not exactly equal,
because of the inevitable approximations in the computation.
One can suggest that there is a pair of complex conjugate spectral values
$$w_{\pm}^1=0.53\ldots \pm i\cdot 0.18\ldots$$
of $\theta$ for which
$\theta (w_{\pm}^1,.)$ has
double zeros $-2.46\ldots \pm i\cdot 7.66\ldots$ and simple zeros
$-2.04\ldots \mp i\cdot 1.47\ldots$ and $-4.7\ldots \pm i\cdot 1.70\ldots$.

In the same way one can consider the truncations $\theta_{\ell}$, $\ell =10$, $11$ -- there are 
values of $q$

$$w_{\pm}^2=0.584\ldots \pm i\cdot 0.062\ldots$$
for which 
$\theta_{\ell}(w_{\pm}^2,.)$ has two zeros
$-10.64\ldots +i\cdot 6.17\ldots$ very close to one another; these are the
fourth and fifth zero in terms of their moduli,
the first three zeros being simple. One can suggest that $\theta$
has spectral values close to $w_{\pm}^2$ with double zeros of the form
$-10.64\ldots +i\cdot 6.17\ldots$.

The truncation $\theta_{10}$ (resp. $\theta_{11}$) has also spectral values
$0.618\ldots \pm i\cdot 0.204\ldots$ (resp.
$0.617\ldots \pm i\cdot 0.204\ldots$) for which $\theta_{10}$ (resp.
$\theta_{11}$) has two zeros of the form $-0.07\ldots +i\cdot 7.34\ldots$
(resp. $-0.07\ldots +i\cdot 7.33\ldots$), the fourth and fifth in terms
of their moduli. It seems likely that $\theta$
has spectral values

$$w_{\pm}^3=0.61\ldots \pm i\cdot 0.20\ldots \not\in \overline{\mathbb{D}_{0.6}}~.$$
One can
suppose that the double zeros of $\theta(w_{\pm}^j,.)$, $j=2$, $3$, are the
fourth and fifth in terms of their moduli. 

\begin{rem}
{\rm The existence of the supposed spectral values $w_{\pm}^j$, $j=1$, $2$, $3$, is not proved, so they are in fact would-be spectral values. In what follows we omit ``would-be''.}
\end{rem}

The spectral values of the truncations $\theta_{\ell}$, $\ell \leq 11$,
suggest that $\theta$ has no spectral values other than $\tilde{q}_1$,
$\tilde{q}_2$, $v_{\pm}$,  $w_{\pm}^1$ and $w_{\pm}^2$ in
the disk $\mathbb{D}_{0.6}$.


\subsection{Spectral values of the truncations of~$\theta$}

In Fig.~\ref{TruncThetaFig} we present the unit circle, 
the spectral values of the
truncation $\theta_{15}$ and the circle of radius $0.309$, the latter  roughly
corresponds to the spectral value~$\tilde{q}_1$. 

\begin{figure}[htbp]
\centerline{\hbox{\includegraphics[scale=0.7]{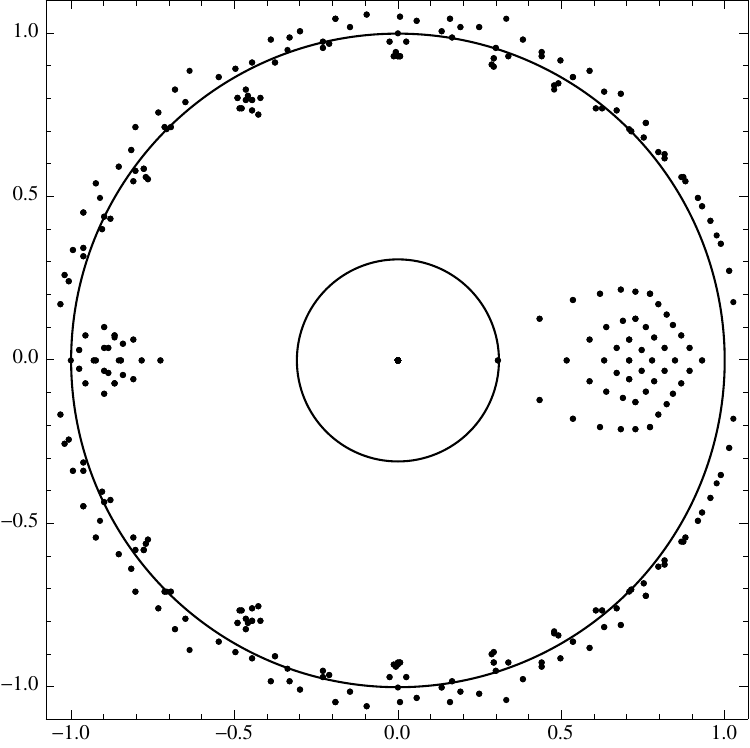}}}
\caption{The spectral points of the
truncation $\theta_{15}$.}
\label{TruncThetaFig}
\end{figure}
The origin is also marked, because it is a
zero of Res$(\theta_{15}(q,x),(\partial \theta_{15}/\partial x)(q,x),x)$; it is not a spectral value of $\theta_{15}$.
Certain spectral values of $\theta_{15}$ lie beyond the unit circle; this is
not the case of the spectral values of $\theta$. One can notice the
accumulation of spectral values around the roots of unity of order $1$,
$\ldots$,~$6$. This is the reason to ask the following

\begin{quest}
  Is it true that every point of the unit circle is an accumulation point
  for the spectrum of~$\theta$?
  \end{quest}
The most notable of these accumulations is the one around $1$. It suggests
that one could expect $\theta$ to have spectral values which are
situated in the right
half-plane as explained by the partially filled matrix below and the comments that follow.

$$\begin{array}{ccccccccc}
  ~&~&~&~&~&~&~&~&(5,6) \\ &&&&&(4,5)&&& \\ &&&(3,4)&&&&(5,6)& \\ 
  &(2,3)&&&(4,5)&&&& \\ 
  (1,2)&&(3,4)&&&&(5,6)&& \\ &(2,3)&&&(4,5)&&&&\\ &&&(3,4)&&&&(5,6)&\\
   &&&&&(4,5)&&&\\ &&&&&&&&(5,6)\end{array}$$
The pairs $(j,j+1)$ mean that when $q$ equals the corresponding spectral value, 
the $j$th and $(j+1)$st zeros of $\theta$ (in terms of the modulus) coalesce;
see Definition~\ref{defiinterms}. The pair $(1,2)$ indicates the position of
the spectral number $\tilde{q}_1$. To its right and in the same horizontal line
are the spectral numbers $\tilde{q}_2$ and $\tilde{q}_3$ (the pairs $(3,4)$
and $(5,6)$ in the middle). The pairs $(2,3)$ indicate the relative
positions of the spectral numbers $v_{\pm}$.

One can suggest that the spectral values with the same pair $(j,j+1)$ are
situated on an arc turned with its concavity towards the point $1$;
this is the case of the three spectral values
$w_{+}^1$, $\tilde{q}_2$, $w_{-}^1$ (pairs $(3,4)$)
and of the four spectral values $w_+^3$, $w_+^2$, $w_-^2$, $w_-^3$
(pairs $(4,5)$),
see the previous subsection. One can also suggest 
that there are couples of arcs (symmetric w.r.t. the horizontal axis)
tending to the point $1$ which consist of
spectral values with pairs $(k,k+1)$, $(k+1,k+2)$, $(k+2,k+3)$, $\ldots$,
$k=1$, $2$, $\ldots$. In Fig.~\ref{TruncThetaFig} this is not true for the
spectral values closest to the point $1$ which is the effect of considering
a truncation of $\theta$ instead of $\theta$ itself.

\section{Proof of Theorem~\protect\ref{tmmain}\protect\label{secprtmmain}}

  $1^{\circ}$. The spectral number $\tilde{q}_1$ is the only one in the disk
  $\mathbb{D}_{0.31}$ (\cite[Theorem~1]{KoDBAN2}) from which and from part (1) 
  of the theorem its part (2) follows. We prove in $2^{\circ}-9^{\circ}$
  part (1) for
  $q\in \mathbb{S}_{0.6,\pi /4 ,\pi /2}=:\mathbb{S}_1$ which implies part (1) for
  $q\in \mathbb{S}_{0.6,3\pi /2 ,7\pi /4}$ due to
  $\theta (\bar{q},\bar{x})=\overline{\theta (q,x)}$. We prove part (1)
  for $q\in \mathbb{S}_{0.6,\pi /2 ,\pi }$ in $10^{\circ}$ from which
  it follows for $q\in \mathbb{S}_{0.6,\pi ,3\pi /2}$. 
\vspace{1mm}

  $2^{\circ}$. We show in $3^{\circ}-7^{\circ}$
  that for $q\in \mathbb{S}_1$ and $|x|=|q|^{-3/2}$,
  one has $|\theta (q,x)|\neq 0$. For $q$ close to $0$ and $k\in \mathbb{N}^*$,
  the equivalences
  $\xi_k\sim -q^{-k}$ hold true (\cite[Theorem~4]{KoBSM1}). The continuous dependence of the zeros $\xi_j$ on the parameter $q$ implies that 
  for $q\in \mathbb{S}_1$ and $|x|=|q|^{-3/2}$, it is true that
  $0<|\xi_1|<|q|^{-3/2}<|\xi_2|$. In $8^{\circ}$ (resp. in $9^{\circ}$)
  we prove that the inequalities
  $0<|\xi_1|<|q|^{-3/2}<|\xi_2|<|q|^{-5/2}$ (resp. all conditions  (\ref{equsepar}))
  hold true for $q\in \mathbb{S}_1$.
\vspace{1mm}

  $3^{\circ}$. Suppose first that $|q|=0.5$. We set
  $q=q(b):=0.5\cdot e^{bi}$, $b\in [\pi /4,\pi /2]$,
  $x=x(a):=0.5^{-3/2}\cdot e^{ai}$, $a\in [0,2\pi]$. The quantity
  $|\theta_3(q(b),x(a))|^2$ is a degree $\leq 18$ trigonometric polynomial
  in the two variables $a$ and $b$ (i.~e. a polynomial of total degree $\leq 18$ in the four quantities $\cos a$, $\sin a$, $\cos b$ and $\sin b$). 
  
  \begin{rem}\label{remcomput}
{\rm To find the minimum of this function of
  two variables  numerically, we use the L-BFGS-B method provided by  the
  SciPy optimization library of Python. This is a Quasi-Newton method type
  with boundary constraints. To garantee the stability, we use 20 random starts.
  The precision is equal to $10^{-5}$. The code execution is fast (about two
  seconds). }\end{rem}
  
  Taking square root we find
  that the minimal value of $|\theta_3(q(b),x(a))|$ is
  $0.45784\ldots >0.45$ (obtained for $a=2.38345$
  and $b=0.78540$; $b$ and $a+b$ are close to $\pi /4$ and $\pi$
  respectively). 
  We set
  
  $$\begin{array}{ccll}\Xi (s,t)&:=&|s^{-1/2}-t^{-1/2}|+|s^{3/2}-t^{3/2}|&
    {\rm and}\\ \\ 
  \Psi (s)&:=&\sum_{j=4}^{\infty}s^{j(j+1)/2-3j/2}=\sum_{j=4}^{\infty}s^{j(j-2)/2}~,&
  s,~t>0~.\end{array}$$ 
For
  $q\in \mathbb{S}_1$ and $|x|=|q|^{-3/2}$, 

$$|\theta_3^{\bullet}(q,x)|\leq \Psi (|q|)=\Psi (0.5)=
0.06827\ldots <0.0683~,$$
  so $|\theta (q,x)|\geq |\theta_3(q,x)|-|\theta_3^{\bullet}(q,x)|>
  0.45-0.0683=0.3817$.
\vspace{1mm}

  $4^{\circ}$. Suppose that $|q|\in [0.5,0.6]$. Consider two different values
  $q_1$ and $q_2$ of $q$, with $|q_1|=0.5$, $q_j\in \mathbb{S}_1$
  and arg$(q_1)=$arg$(q_2)$.
  Consider the
  quantities $\tau_j:=\theta_3(q_j,x_j)$, where $x_j:=|q_j|^{-3/2}\omega$,
  with $|\omega |=1$, $j=1$, $2$. Hence for all $j$, the monomials
  $q_1^{j(j+1)/2}x_1^j$ and $q_2^{j(j+1)/2}x_2^j$ have the same argument.
  Both quantities $\tau_1$ and $\tau_2$ have monomials equal to $1$
  and their monomials $(q_j)^3(x_j)^2$
  are equal and of modulus $1$. One has 

  $$\begin{array}{ccl}|\tau_2-\tau_1|&\leq&||q_1|\cdot |x_1|-
  |q_2|\cdot |x_2||+||q_1|^6\cdot |x_1|^3-
  |q_2|^6\cdot |x_2|^3|\\ \\ &=&\Xi (|q_2|,|q_1|)~\leq ~
  \Xi (0.6,0.5)~=~0.2344\ldots ~<~0.2345~.
  \end{array}$$
  Set
  $\lambda_j:=\theta_3^{\bullet}(q_j,x_j)$, $j=1$, $2$. Clearly

  $$|\theta (q_2,x_2)|\geq |\theta (q_1,x_1)|-H~,~~~\,
  H:=|\theta (q_1,x_1)-\theta (q_2,x_2)|~,$$ 
  hence $H\leq |\tau_2-\tau_1|+|\lambda_1-\lambda_2|$. Observe that for $j=4$, $5$, $\ldots$, one has 
  
  $$\begin{array}{l}|q_i^{j(j+1)/2}x_i^j|=|q_i^{j(j-2)/2}|~,~ i=1,~2,~~~\, {\rm so}~~~\,  |q_2^{j(j+1)/2}x_2^j|\geq |q_1^{j(j+1)/2}x_1^j|~~~\, {\rm and} \\ \\ |q_1^{j(j+1)/2}x_1^j-q_2^{j(j+1)/2}x_2^j|=|q_2^{j(j+1)/2}x_2^j|-|q_1^{j(j+1)/2}x_1^j|~.
     \end{array}$$
 As

  $$\begin{array}{ccl}|\lambda_1-\lambda_2|&\leq&\sum_{j=4}^{\infty}
  |q_1^{j(j+1)/2}x_1^j-q_2^{j(j+1)/2}x_2^j|=\Psi (|q_2|)-\Psi (|q_1|)\\ \\ 
  &\leq&\Psi (0.6)-\Psi (0.5)~=~
  0.0853\ldots ~<~0.0854~,\end{array}$$ 
  one gets $|\theta (q_2,x_2)|\geq 0.3817-0.2345-0.0854=0.0618>0$. 
\vspace{1mm}

  $5^{\circ}$. Suppose that $|q|\in [0.4,0.5]$. We use the notation from
  $4^{\circ}$, but this time $|q_2|\in [0.4,0.5]$. Hence  

  $$\begin{array}{ccccl}|\tau_2-\tau_1|&=&\Xi (|q_2|,|q_1|)&\leq &\Xi (0.5,0.4)
  =0.2674\ldots
  <0.2675~~~\, \, 
  {\rm and}\\ \\ |\lambda_1-\lambda_2|&=&\Psi (|q_1|)-\Psi (|q_2|)&\leq &
  \Psi (0.5)-\Psi (0.4)
  =0.0416\ldots <0.0417~.\end{array}$$
  Thus $H\leq 0.2675+0.0417=0.3092$ and
  $|\theta (q_2,x_2)|\geq 0.3817-0.3092=0.0725>0$.
\vspace{1mm}

  $6^{\circ}$. Suppose that $|q|=0.36$. Similarly to $3^{\circ}$ we
  set $q=0.36\cdot e^{bi}$, $b\in [\pi /4,\pi /2]$, $x=0.36^{-3/2}e^{ai}$,
  $a\in [0,2\pi]$, and we find that the minimal value of
  $|\theta_3(q(b),x(a))|$ is $0.55011\ldots >0.55$ (obtained for $a=3.44113$
  and $b=0.78540$). Next,

  $$|\theta_3^{\bullet}(q,x)|~\leq ~\Psi (|q|)
  ~\leq ~\Psi (0.36)~
  =~0.01727\ldots ~<~0.01718$$
  and $|\theta (q,x)|\geq |\theta_3(q,x)|-|\theta_3^{\bullet}(q,x)|>
  0.55-0.01718=0.53282$.
  \vspace{1mm}
  
  $7^{\circ}$. Suppose that $|q|,|q_2|\in [0.36,0.4]$ and $|q_1|=0.36$.
  As in $4^{\circ}$ we find that
  $$\begin{array}{ccl}|\tau_2-\tau_1|&\leq &\Xi (0.4,0.36)
  =0.1225\ldots <0.1226~~~\, {\rm and}\\ \\ 
  |\lambda_1-\lambda_2|&\leq &\Psi (0.4)-\Psi (0.36)=
  0.00938\ldots <0.00939~,\end{array}$$
  so $H\leq 0.1226+0.00939=0.13199$ and
  $|\theta (q_2,x_2)|\geq 0.53282-0.13199=0.40083>0$.

  Suppose that
  $|q|,|q_2|\in [0.24,0.36]$ and $|q_1|=0.36$. Hence
  $$\begin{array}{ccl}|\tau_2-\tau_1|&\leq& \Xi (0.36,0.24)=
  0.4729\ldots <0.473~~~\, {\rm and}\\ \\ 
  |\lambda_1-\lambda_2|&\leq& \Psi (0.36)-\Psi (0.24)=
  0.01393\ldots <0.01394~.\end{array}$$
  Thus $H\leq 0.473+0.01394=0.48694$ and
  $|\theta (q_2,x_2)|\geq 0.53282-0.48694=0.04588>0$.

  For $|q|\in (0,0.24)$, $q\in \mathbb{S}_1$, $|x|=|q|^{-3/2}$,
  we use the identity
  $$\theta_3=1+qx+q^3x^2+q^6x^3=(1+q^2x)(1+(1-q)qx+q^4x^2)~.$$
  One has
  
  1) $|q^2x|=|q|^{1/2}<0.24^{1/2}<0.25^{1/2}=0.5$, so $|1+q^2x|>0.5$; 
  
  2) $|q^4x^2|=|q|< 0.24$;
  
  3) $|1-q|$ takes its minimal value for $q=0.24\cdot e^{\pi i/4}$
  (to be checked directly), this value is $0.8474\ldots >0.8474$; 
  
  4)
  $|qx|=|q|^{-1/2}>0.24^{-1/2}>0.25^{-1/2}=2$. Thus

  $$|\theta_3|> 0.5\cdot (2\cdot 0.8474-1-0.24)=0.2274>
  \sum_{j=4}^{\infty}0.24^{j(j-2)/2}=0.0033\ldots \geq |\theta_3^{\bullet}|~,$$
  so $|\theta |\geq |\theta_3|-|\theta_3^{\bullet}|>0.2274-0.0034=0.224>0$.
  One concludes that for $q\in \mathbb{S}_1$ and $|x|=|q|^{-3/2}$, 

  $$|\theta |>\min ( 0.3817,0.0618,0.0725,0.53282,
  0.40083,0.04588,0.224)=0.04588~.$$

  $8^{\circ}$. Suppose that $q\in \mathbb{S}_1$ and $|x|=|q|^{-5/2}$. For the
  values of $|q|$ (denoted by $r_{\nu}:=0.2025+0.0025\cdot \nu$, $\nu =0$, $\ldots$, $159$, $r_{159}=0.6$)
  we compute the minimal value of
  $|\theta_5(q,x)|$ (as this was done for $|\theta_3|$ for $|q|=0.5$ and
  $|q|=0.36$). The corresponding minimal values
  $\mu_{\nu}$ belong to the interval $[0.74930884\ldots ,7.60808216\ldots ]$.

  Suppose that $|q|\in [r_{\nu},r_{\nu +1}]$.
  We consider two values of $q$ (denoted
  as in $3^{\circ}-7^{\circ}$ by $q_1$ and $q_2$, with $|q_1|=r_{\nu}$ and
  arg$(q_1)=$arg$(q_2)$) and two values of $x$ ($x_1$ and $x_2$, with
  $|x_j|=|q_j|^{-5/2}$ and arg$(x_1)=$arg$(x_2)$). We set

  $$K(s,t):=2|s^{-3/2}-t^{-3/2}|+|s^{-2}-t^{-2}|+|s^{5/2}-t^{5/2}|~,~~~\,
  L(s):=
  \sum_{j=6}^{\infty}s^{j(j-4)/2}~.$$
  Recall that $\theta_5(q,x)=1+qx+q^3x^2+q^6x^3+q^{10}x^4+q^{15}x^5$.
  Hence the quantities $\theta_5(q_1,x_1)$ and $\theta_5(q_2,x_2)$ both  contain 
  the monomial $1$, and their monomials $(q_j)^{10}(x_j)^4$ are also equal and of
  modulus~1. Clearly 
$$\begin{array}{cccc}|\theta_5(q_1,x_1)-\theta_5(q_2,x_2)|&\leq& 
    K(r_{\nu},r_{\nu +1})~,&\\ \\ 
  |\theta_5^{\bullet}(q_1,x_1)-\theta_5^{\bullet}(q_2,x_2)|&\leq&
  |L(r_{\nu})-L(r_{\nu +1})|&{\rm 
  and}\end{array}$$
  
  \begin{equation}\label{equrho}\begin{array}{ccl}|\theta (q_2,x_2)|&\geq&
    |\theta (q_1,x_1)|-K(r_{\nu +1},r_{\nu})-|L(r_{\nu})-L(r_{\nu +1})|\\ \\ &\geq&
    \min (\mu_{\nu},\mu_{\nu +1})-K(r_{\nu +1},r_{\nu})-|L(r_{\nu})-L(r_{\nu +1})|=:
    \rho_{\nu}~.
  \end{array}\end{equation}
  For all values of $\nu$, one has $\rho_{\nu}>0$ which implies
  $0<|\xi_1|<|q|^{-3/2}<|\xi_2|<|q|^{-5/2}$. In fact, we prove not
  (\ref{equrho}), but the stronger inequality (\ref{equeta}).
\vspace{1mm}

$9^{\circ}$. In order to prove the conditions  (\ref{equsepar}) for $k\geq 3$ 
we introduce the quantity $M(s):=L(s)+s^{5/2}=\sum_{j=5}^{\infty}s^{j(j-4)/2}$. We
remind that the functions $\Theta^*$ and $G$ were defined in
(\ref{equTheta}) and (\ref{equThetaG}). 

  \begin{lm}\label{lmM}
    It is true that for $|x|=|q|^{-k+1/2}$, $k\geq 3$, one has
    $|G(q,x)|\leq M(|q|)$
    and $|G(q,x/q)|\leq M(|q|)$.
    \end{lm}
  \begin{proof}
    Indeed, 
    
    $$\begin{array}{ccccl}|G(q,x)|&\leq&
      \sum_{j=1}^{\infty}|q|^{j(j-1)/2+(2k-1)j/2}&\leq&
      \sum_{j=1}^{\infty}|q|^{j(j-1)/2+5j/2}\\ \\ &=&
    \sum_{j=1}^{\infty}|q|^{j(j+4)/2}&=&M(|q|)~~~\, {\rm and}\\ \\ 
    |G(q,x/q)|&\leq& \sum_{j=1}^{\infty}|q|^{j(j-1)/2+(2k+1)j/2}&\leq&
    \sum_{j=1}^{\infty}|q|^{j(j-1)/2+7j/2}\\ \\ &=&
    \sum_{j=1}^{\infty}|q|^{j(j+6)/2}&\leq&M(|q|)~.\end{array}$$
    \end{proof}
  
  It is checked numerically that for all $\nu \in [0..159]$, one has

  \begin{equation}\label{equeta}\eta_{\nu}:=\min (\mu_{\nu},\mu_{\nu +1})-
  K(r_{\nu +1},r_{\nu})-|L(r_{\nu})-L(r_{\nu +1})|-
  2M(r_{\nu +1})>0~.\end{equation} 
  Hence $|\theta (q_2,x_2)|>2M(r_{\nu +1})$, see (\ref{equrho}). As
  
  $$|\theta (q_2,x_2)|\geq |\Theta^*(q_2,x_2)|-|G(q_2,x_2)|~~~\, {\rm and}~~~\, 
  |G(q_2,x_2)|\leq M(r_{\nu +1})~,$$
  this implies
  $|\Theta^*(q_2,x_2)|>M(r_{\nu +1})\geq |G(q_2,x_2)|$.

\begin{lm}\label{lmTheta*}
  Set $|q|=r\in (0,0.6]$.
  Suppose that $|x|=r^{-k+1/2}$, $k\in \mathbb{N}$,
  $k\geq 3$, and 
  that $|\Theta^*(q,x)|>M(r)\geq |G(q,x)|$.
  Then $|\Theta^*(q,x/q)|>M(r)\geq |G(q,x/q)|$.
\end{lm}

\begin{proof}
  One checks directly that
  $$\Theta^*(q,x/q)=
  \Theta^*(q,x)\cdot (1+x)/(1+1/x)=\Theta^*(q,x)\cdot x~,$$
  so
  $|\Theta^*(q,x/q)|=|\Theta^*(q,x)|\cdot |x|>|\Theta^*(q,x)|>M(r)\geq
  |G(q,x/q)|$, see Lemma~\ref{lmM}.

\end{proof}

The lemma implies that if $0<|q|\leq 0.6$ and $|\theta (q,x)|>0$, then $|\theta (q,x/q)|>0$.
For fixed $q$, if $x$ runs over the circle of radius $|q|^{-k+1/2}$, then
$x/q$ runs over the circle of radius $|q|^{-k-1/2}$. For $k\geq 3$,
one can apply successively Lemma~\ref{lmTheta*} to conclude that
$\theta \neq 0$ for $|x|=|q|^{-k+1/2}$ and $q\in \mathbb{S}_1$. As for $|q|<c_0$,
one has (\ref{equsepar}) and as the zeros of $\theta$ depend continuously
on $q$, conditions (\ref{equsepar}) with $k\geq 3$ hold true
for $q\in \mathbb{S}_1$; for $k=1$ and $2$, this was proved in
$3^{\circ}-7^{\circ}$ and $8^{\circ}$ respectively.
\vspace{1mm}

$10^{\circ}$. It is shown in \cite{KoVJM} that for 
$q\in \mathbb{S}_{0.55,\pi /2,3\pi /2}$, conditions (\ref{equsepar}) hold true;
and that for $q\in \mathbb{S}_{0.6,\pi /2,3\pi /2}$, conditions
(\ref{equsepar}) hold true for $k\neq 2$,~$3$, there is a simple zero of
$\theta$ in the annulus $\mathbb{A}_{5/2,7/2}(|q|)$ and there are exactly two
simple zeros in $\mathbb{A}_{3/2,7/2}(|q|)$. We show below that for
$q\in \mathbb{S}_{0.6,\pi /2,\pi }$ (hence for $q\in \mathbb{S}_{0.6,\pi /2,3\pi /2}$),
condition (\ref{equrho}) takes place.

Using the continuous dependence of the zeros $\xi_j$ on $q$, this is
sufficient to conclude that for $q\in \mathbb{S}_{0.6,\pi /2,3\pi /2}$,
the zeros of $\theta$ are strongly separated in modulus. Indeed, conditions
(\ref{equsepar}) hold true for $q\in \overline{\mathbb{D}_{c_0}}$, see \cite[Lemma~1]{KoAM}.
Conditions (\ref{equrho}) and conditions (\ref{equsepar}) with $k=1$ mean that
for $q\in \mathbb{S}_{0.6,\pi /2,3\pi /2}$, there is no zero of $\theta$ with
$|x|=|q|^{-3/2}$ and $|x|=|q|^{-5/2}$. The zero $\xi_2$ depending continuously on
$q$, condition (\ref{equsepar}) holds true for $k=1$ and $2$ and
$q\in \mathbb{S}_{0.6,\pi /2,3\pi /2}$.
Hence for $q\in \mathbb{S}_{0.6,\pi /2,3\pi /2}$, $\xi_2\in \mathbb{A}_{3/2,5/2}(|q|)$
and $\xi_2$ is the only zero of $\theta$ in $\mathbb{A}_{3/2,5/2}(|q|)$. 

The zero $\xi_3$ is the only zero of $\theta$ in $\mathbb{A}_{5/2,7/2}(|q|)$ for
$|q|\leq c_0$, see \cite[Lemma~1]{KoAM}. For
$q\in \mathbb{S}_{0.6,\pi /2,3\pi /2}$, $\xi_2$ and $\xi_3$ are
the only zeros of $\theta$ in 
$\mathbb{A}_{3/2,7/2}(|q|)$, they are simple (see \cite{KoVJM}) and $\xi_2$
is the only zero in 
$\mathbb{A}_{3/2,5/2}(|q|)$, so $\xi_3$ is the only zero in 
$\mathbb{A}_{5/2,7/2}(|q|)$. Hence conditions (\ref{equsepar}) hold true for
$q\in \mathbb{S}_{0.6,\pi /2,3\pi /2}$, i.e. the zeros of $\theta$ are strongly
separated in modulus. 

The proof that for $q\in \mathbb{S}_{0.6,\pi /2,\pi }$, condition (\ref{equrho})
takes place, is given numerically, as the previous such proofs, for
$|q|=0.55+0.005j$, $j=0$,$\ldots$, $10$. In all cases one obtains $\rho_{\nu}>0$. The quantity $|\theta_5(q,x)|^2$ (whose minimal values $\mu_{\nu}$ are used in (\ref{equrho})) is a polynomial in $\cos a$, $\sin a$, $\cos b$ and $\sin b$ of total degree $\leq 40$.

\end{document}